\documentclass[11pt,reqno]{amsart}

\usepackage[T1]{fontenc}
\usepackage[utf8]{inputenc}
\usepackage[english]{babel}
\usepackage{lmodern}
\usepackage{amsmath,amssymb,amsthm,mathtools,mathrsfs}
\usepackage{microtype}
\usepackage{enumitem}
\usepackage{xcolor}
\usepackage[colorlinks=true,linkcolor=blue!50!black,citecolor=blue!50!black,urlcolor=blue!50!black]{hyperref}
\hypersetup{
  pdftitle={Smooth Diffeomorphisms and Mahler's Problem on Liouville Numbers},
  pdfauthor={Diego Marques}
}

\numberwithin{equation}{section}
\newtheorem{theorem}{Theorem}[section]
\newtheorem{proposition}[theorem]{Proposition}
\newtheorem{lemma}[theorem]{Lemma}
\newtheorem{corollary}[theorem]{Corollary}
\theoremstyle{definition}
\newtheorem{definition}[theorem]{Definition}
\theoremstyle{remark}
\newtheorem{remark}[theorem]{Remark}

\newcommand{\R}{\mathbb R}
\newcommand{\Q}{\mathbb Q}
\newcommand{\Z}{\mathbb Z}
\newcommand{\C}{\mathbb C}
\newcommand{\N}{\mathbb N}
\newcommand{\Ralg}{\overline{\mathbb Q}_{\mathbb R}}
\newcommand{\Lio}{\mathscr L}
\newcommand{\id}{\operatorname{id}}

\newcommand{\NA}{\operatorname{NA}}

\newcommand{\norm}[1]{\left\lVert #1\right\rVert}
\newcommand{\abs}[1]{\left|#1\right|}
\newcommand{\set}[1]{\left\{#1\right\}}

\title[Smooth Diffeomorphisms and Mahler's Problem]
{Smooth Diffeomorphisms and Mahler's Problem on Liouville Numbers}

\author{Diego Marques}
\address{Departamento de Matem\'atica, Universidade de Bras\'ilia, Bras\'ilia, 70910-900, Brazil}
\email{diego@mat.unb.br}


\begin{document}

\begin{abstract}
A classical theorem of Maillet asserts that every nonconstant rational
function over $\Q$ maps Liouville numbers to Liouville numbers. In 1984,
Mahler asked whether a transcendental entire function can have the same
property. We prove a strong smooth counterpart: writing $\Lio$ for the set
of Liouville numbers, there exist orientation-preserving $C^\infty$
diffeomorphisms $f:\R\to\R$, arbitrarily close to the identity and
transcendental over $\R(x)$, such that for every real number field
$K\subset\R$, every $n\geq1$, and every $m\geq0$,
\[
D^m(f^{\circ n})(K)\subseteq K,
\qquad
D^m(f^{\circ n})(\Lio)\subseteq\Lio.
\]
In fact, the non-analyticity locus may be prescribed as any nonempty compact perfect nowhere-dense set disjoint from the real algebraic and Liouville numbers. The proof combines Maillet's theorem with an arithmetic refinement
of K\"orner's smooth polynomial sewing method and a rational-germ construction.
\end{abstract}

\maketitle

\section{Introduction}

In 1984, Mahler~\cite{Mahler1984} asked whether there exists a transcendental
entire function that maps every Liouville number to a Liouville number. The
question is motivated by a classical theorem of Maillet~\cite{Maillet1906},
which shows that every nonconstant rational function $R\in\mathbb Q(x)$ has
this property. Mahler's problem therefore asks whether this arithmetic
stability, automatic in the rational setting, can persist under genuine
transcendence in the strongest analytic category.

Following recent terminology, we say that a real function has
\emph{Maillet's property} if it maps every Liouville number in its domain to
a Liouville number, and we denote the set of Liouville numbers by $\Lio$.
Thus Mahler's problem is precisely the question of whether Maillet's property
can be realized by a transcendental entire function. It remains open; see
Schleischitz~\cite{Schleischitz2025} for a recent formulation and further
context.

Mahler's question belongs to the broader tradition of arithmetic
interpolation. Classical constructions of transcendental analytic functions
with prescribed arithmetic values go back to
Weierstrass~\cite{Weierstrass1923} and
St\"ackel~\cite{Stackel1895,Stackel1902}. A particularly striking theorem of
van der Poorten~\cite{vanderPoorten1968} produces transcendental entire
functions whose values, together with those of all their derivatives,
preserve every number field. Mahler's problem is different in nature: the
set $\Lio$ is defined by an exceptional Diophantine approximation property,
rather than by algebraic constraints on prescribed values. In this direction,
Marques and Moreira~\cite{MarquesMoreira2015} constructed transcendental
entire functions preserving a large distinguished subclass of the Liouville
numbers, while Marques and Schleischitz~\cite{MarquesSchleischitz2016}
obtained simultaneous preservation results for broad parametrized classes.
Thus the central difficulty is not merely to prescribe many arithmetic
values of a transcendental function, but to preserve the entire Diophantine
class $\Lio$.

\subsection*{The regularity threshold and a recent finite-order construction}

The present paper addresses the regularity threshold behind Mahler's
problem. Recent work of Lelis, Moreira, and Silva
\cite[Theorem~1.2 and Corollary~1.3]{LMS2026} shows that Maillet's
phenomenon persists at every finite level of smoothness. More precisely,
for every integer $k\geq1$ and every real number $t>2k$, their construction
produces an uncountable dense family of transcendental functions
$f\in C^k(\R)$ satisfying
\[
f(\Q)\subseteq\Q,
\qquad
f(\Lio)\subseteq\Lio,
\]
together with polynomial denominator control. Their work is
used here only to place the present result in the context of the known
finite-order theory; none of the results or proofs below depend on it.

The dependence of the admissible exponent on $k$ makes explicit the
finite-order nature of that construction. In particular, the results for
successive values of $k$ do not by themselves provide a compatible passage
to a single $C^\infty$ function. This leads naturally to the question of
whether infinite smoothness is compatible with Maillet's property by a
different mechanism. The construction developed here gives an affirmative
answer: rather than seeking denominator estimates uniform in the order of
differentiation, it separates the arithmetic locus from the
non-analyticity locus and replaces global height control by exact local
rationality.

\subsection*{Main result}

Our main theorem shows that the derivative loss inherent in the
finite-order construction does not reflect an obstruction in the smooth
category itself. Crucially, our $C^\infty$ result is not obtained by taking
a limit of finite-$C^k$ constructions of the kind just discussed: it
bypasses derivative loss entirely by separating the arithmetic locus from
the whole non-analyticity locus of the map. Infinite smoothness still
permits a remarkably strong form of arithmetic flexibility.

\begin{theorem}\label{thm:headline}
There exists an orientation-preserving $C^\infty$ diffeomorphism $f:\R\to\R$, transcendental over $\R(x)$, such that for every real number field $K\subset\R$, every $n\geq1$, and every $m\geq0$,
\[
D^m(f^{\circ n})(K)\subseteq K,
\qquad
D^m(f^{\circ n})(\Lio)\subseteq\Lio.
\]
Moreover, such diffeomorphisms may be chosen arbitrarily close to the identity in the compact-open $C^\infty$ topology.
\end{theorem}

Already the case $m=0$ and $n=1$ yields a transcendental smooth function
satisfying
\[
f(\Lio)\subseteq\Lio.
\]
Thus Mahler's problem has a strong affirmative counterpart in the smooth
category. The conclusion is substantially stronger than the mere preservation of
Liouville numbers: a single map simultaneously preserves every real number
field and the full Liouville set, and both invariances survive under every
derivative of every positive iterate.

\subsection*{The separation principle}

The distinction between $C^\infty$ smoothness and analyticity is essential
to our construction. A direct extension of the finite-order approach would
seek arithmetic estimates compatible with derivatives of arbitrarily high
order. We take a different route, based on a \emph{geometric separation
principle}: all non-analyticity is confined to a perfect set containing no
arithmetic point that the theorem needs to control. Away from this set, the
map is locally an element of $\Q(x)$. Maillet's theorem can therefore be
applied locally, while transcendence is forced globally by preventing the
different rational germs from analytically continuing into one another.

Schleischitz~\cite{Schleischitz2025} observed that piecewise-defined
continuous functions which are locally rational over $\Q$ already have
Maillet's property. The new difficulty here is to realize this principle in
class $C^\infty$, while simultaneously enforcing global transcendence and
ensuring that the relevant rational germs remain nonconstant under
differentiation and iteration. The construction addresses these requirements
by prescribing exactly the non-analyticity locus and by introducing a
rational perturbation whose nonreal poles provide a persistent certificate
of nonconstancy.

More precisely, Theorem~\ref{thm:cloaking} gives a
prescribed-singularity form of the main result. Given any nonempty compact
perfect nowhere-dense set $E$ disjoint from the real algebraic and Liouville
numbers, the diffeomorphism may be chosen to be non-analytic exactly on $E$
and locally rational over $\Q$ on every component of $\R\setminus E$.

The technical core is the arithmetic polynomial sewing theorem,
Theorem~\ref{thm:sewing}. For every nonempty compact perfect nowhere-dense set
$E\subset\R$ and every dense subfield $F\subset\R$, it produces a function
$g\in C_c^\infty(\R)$ which agrees on each connected component of
$\R\setminus E$ with a polynomial in $F[x]$, while being non-analytic
exactly on $E$. The essential additional feature relative to the underlying
sewing construction is the exact coefficient restriction. It cannot be
imposed by a terminal approximation; rather, it is incorporated into the
induction by fixing the polynomial assigned to one new gap at each stage,
while the remaining intervals accommodate the resulting highest-order jet
discrepancy.

To prove Theorem~\ref{thm:headline}, we choose $E$ containing neither real
algebraic nor Liouville numbers and apply Theorem~\ref{thm:sewing} with
$F=\Q$. A fixed rational perturbation with nonreal poles then converts the
local polynomial models into rational ones whose relevant compositions and
derivatives remain nonconstant. Maillet's theorem applies along Liouville
orbits, while evaluation of rational functions over $\Q$ preserves every
real number field. The prescribed non-analyticity, on the other hand,
forces global transcendence.

The sewing mechanism itself is not specific to Liouville numbers. More
generally, the same separation philosophy may apply whenever the arithmetic
set under consideration is suitably invariant under the relevant rational
maps and can be kept disjoint from the prescribed non-analyticity locus.
We return briefly to this point and to the analytic barrier in the
concluding remarks.

\subsection*{Organization}

The paper is organized as follows. Section~\ref{sec:prelim} records the
arithmetic and regularity facts used throughout. Section~\ref{sec:invisible}
constructs arithmetically invisible perfect sets. Section~\ref{sec:sewing}
proves the arithmetic polynomial sewing theorem. Section~\ref{sec:germs}
develops the rational-germ and pole mechanism. Section~\ref{sec:construction}
proves the prescribed-singularity theorem and then
Theorem~\ref{thm:headline}. Finally, Section~\ref{sec:concluding} discusses the scope and limitations of the method.

\section{Preliminaries}\label{sec:prelim}

Throughout, $\Ralg:=\overline{\Q}\cap \R$ denotes the set of real algebraic numbers. A real number $\xi$ is a \emph{Liouville number} if, for every integer $N\geq1$, there are integers $p$ and $q\geq2$ such that
\[
0<\abs{\xi-\frac pq}<q^{-N}.
\]
By a \emph{real number field} we mean a finite extension $K/\Q$ contained in $\R$.
For a smooth function $u$, we write $D^m u$ for its $m$th derivative, with $D^0u=u$, and $f^{\circ n}$ for the $n$th iterate of a map $f$. We denote by $\NA(u)$ the set of points at which $u$ is not real-analytic. A function $u:J\to\R$ is called \emph{algebraic over $\R(x)$} if
$P(x,u(x))=0$ on $J$ for some nonzero $P\in\R[X,Y]$; otherwise it is
\emph{transcendental over $\R(x)$}.

We endow $C^\infty(\R)$ with its usual compact-open Fr\'echet topology, generated by the seminorms
\[
p_{A,r}(u)
:=\max_{0\leq j\leq r}\sup_{|x|\leq A}\abs{u^{(j)}(x)},
\qquad A,r\in\N.
\]
For a compact interval $J$, we write
\[
\norm{u}_{C^r(J)}
:=\max_{0\leq j\leq r}\sup_{x\in J}\abs{u^{(j)}(x)}.
\]

The arithmetic tool is the following classical theorem.

\begin{theorem}[Maillet {\cite[Chap.~III, Th\'eor\`eme~I$_3$
and Corollaire, pp.~32--33]{Maillet1906}}]\label{thm:maillet}
Let $R\in\Q(x)$ be nonconstant. If $\xi\in\Lio$, then $R(\xi)\in\Lio$.
\end{theorem}

The value $R(\xi)$ is automatically well-defined, since every pole of
a rational function in $\Q(x)$ is algebraic, whereas every Liouville
number is transcendental.

We shall also use the classical fact that a $C^\infty$ algebraic
function of one real variable is necessarily real-analytic; see
Koll\'ar \cite[p.~312, Remark following Definition~9]{Kollar2017}.
See also Bochnak--Coste--Roy \cite[\S8.1]{BCR1998}.

\begin{theorem}[Smooth algebraic functions are analytic]\label{thm:smooth-algebraic}
Let $J\subset\R$ be an open interval and let $u\in C^\infty(J)$. If there exists a nonzero polynomial $P\in\R[X,Y]$ such that
\[
P(x,u(x))=0
\qquad(x\in J),
\]
then $u$ is real-analytic on $J$.
\end{theorem}

With these arithmetic and regularity facts in place, we now turn to the geometric ingredient of the construction.

\section{An arithmetically invisible perfect set}\label{sec:invisible}

We shall need a compact perfect set containing neither real algebraic nor
Liouville numbers. In fact, the exceptional set can be chosen metrically
very large. This strengthens the geometric separation underlying the
construction: the eventual diffeomorphism may fail to be analytic on
almost all of a prescribed interval while remaining locally rational at
every algebraic and every Liouville point.

\begin{proposition}[Large invisible perfect sets]\label{prop:invisible}
For every $0<\eta<1$, there exists a nonempty compact perfect
nowhere-dense set $E\subset(0,1)$ such that
\[
E\cap(\Ralg\cup\Lio)=\varnothing
\]
and
\[
\lambda(E)>1-\eta,
\]
where $\lambda$ denotes Lebesgue measure.
\end{proposition}

\begin{proof}
The set $\Ralg$ is countable, while the set $\Lio$ of Liouville numbers
has Lebesgue measure zero. Hence
\[
X:=(0,1)\setminus(\Ralg\cup\Lio)
\]
has full Lebesgue measure in $(0,1)$. By inner regularity of Lebesgue
measure, there exists a compact set $K\subset X$ such that
\[
\lambda(K)>1-\eta.
\]

By the Cantor--Bendixson theorem, $K$ admits a decomposition
\[
K=E\cup C,
\]
where $E$ is perfect and $C$ is countable. Since $C$ has measure zero,
\[
\lambda(E)=\lambda(K)>1-\eta.
\]
In particular, $E$ is nonempty. Moreover,
\[
E\subset K\subset X,
\]
and therefore
\[
E\cap(\Ralg\cup\Lio)=\varnothing.
\]

Finally, $\Q\subset\Ralg$ is dense in $\R$. Since $E$ contains no
algebraic number, it cannot contain a nonempty open interval. Thus $E$
has empty interior; being closed, it is nowhere dense.
\end{proof}

\begin{remark}\label{rem:badly-approximable}
For the main construction, only
\[
E\cap(\Ralg\cup\Lio)=\varnothing
\]
is needed. Alternatively, one may choose $E$ to consist entirely of badly approximable numbers. Indeed, the continued-fraction Cantor set
\[
\mathcal C_2
=
\set{[0;a_1,a_2,\ldots]:a_j\in\{1,2\}\ \text{for all }j\geq1}
\]
is compact and perfect, and satisfies
\[
\abs{x-\frac pq}>\frac1{4q^2}
\qquad
(x\in\mathcal C_2,\ p\in\Z,\ q\geq1);
\]
see Khinchin \cite[Chapters~I--II]{Khinchin1997}. Removing the countable set $\Ralg$ by a standard perfect-set extraction yields a nonempty compact perfect set
\[
E\subset\mathcal C_2\setminus\Ralg
\]
with the same estimate.
\end{remark}

\section{Arithmetic polynomial sewing}\label{sec:sewing}

Körner \cite{Korner2007} characterized the closed sets on which a smooth function may fail to be locally polynomial.  We require an arithmetic refinement: the local polynomials must lie in a prescribed dense subfield.  We give a complete proof because this coefficient restriction is the main arithmetic ingredient of the construction.

\subsection*{Jet interpolation}

We begin with a quantitative form of Hermite interpolation that will control the lower-order endpoint data throughout the sewing process.

\begin{lemma}[Hermite interpolation with norm control]\label{lem:hermite}
Let $r\geq0$ and let $J=[a,b]$ with $a<b$. Then there exists a constant $C=C(r,J)$ such that, for arbitrary real numbers $\alpha_0,\dots,\alpha_r$ and $\beta_0,\dots,\beta_r$, one can find a polynomial $H$ of degree at most $2r+1$ satisfying
\[
 H^{(j)}(a)=\alpha_j,
 \qquad
 H^{(j)}(b)=\beta_j
 \qquad(0\leq j\leq r),
\]
and
\[
 \norm{H}_{C^r(J)}
 \leq C\max_{0\leq j\leq r}\{ |\alpha_j|,|\beta_j|\}.
\]
\end{lemma}

\begin{proof}
Let $\mathcal P_{2r+1}$ be the real vector space of polynomials of degree at most $2r+1$ and consider the linear map
\[
T:\mathcal P_{2r+1}\longrightarrow\R^{2r+2},
\qquad
T(H)
=\bigl(H(a),\dots,H^{(r)}(a),H(b),\dots,H^{(r)}(b)\bigr).
\]
If $T(H)=0$, then $(x-a)^{r+1}(x-b)^{r+1}$ divides $H$, which is impossible unless $H=0$ because the divisor has degree $2r+2$. Thus $T$ is injective, and dimension equality makes it an isomorphism. Equip $\mathcal P_{2r+1}$ with the $C^r(J)$ norm and $\R^{2r+2}$ with the maximum norm. Since $T^{-1}$ is a linear map between finite-dimensional normed spaces, it is continuous, and hence there is a constant $C=C(r,J)$ such that
\[
\|H\|_{C^r(J)}
\leq C\|T(H)\|_\infty
=
C\max_{0\leq j\leq r}\{|H^{(j)}(a)|,|H^{(j)}(b)|\}.
\]
This gives the required estimate.
\end{proof}

The next lemma is the device that makes the induction possible: an arbitrary derivative of one new order can be imposed at the endpoints at arbitrarily small cost in all lower derivatives.

\begin{lemma}[Small correction of a higher jet]\label{lem:high-jet}
Let $s\geq1$, let $J=[a,b]$ be a nondegenerate compact interval,
let $A,B\in\R$, and let $\eta>0$. Then there exists a polynomial $Q$ such that
\begin{align*}
Q^{(j)}(a)=Q^{(j)}(b)&=0 &&(0\leq j<s),\\
Q^{(s)}(a)&=A, & Q^{(s)}(b)&=B,
\end{align*}
and
\[
\norm{Q}_{C^{s-1}(J)}<\eta.
\]
\end{lemma}

\begin{proof}
Set $t=(x-a)/(b-a)$. It suffices to work on $[0,1]$ with order-$s$
endpoint data $(b-a)^sA$ and $(b-a)^sB$, since the affine change of
variables rescales derivatives only by fixed powers of $b-a$. Thus, after
relabeling the scaled data as $A$ and $B$, let $N>s$ and define
\[
L_N(t)=\frac{A}{s!}t^s(1-t)^N,
\qquad
R_N(t)=\frac{B}{s!}(t-1)^st^N,
\qquad
Q_N=L_N+R_N.
\]
The endpoint factors give exactly the prescribed jets. It remains to show
that $Q_N\to0$ in $C^{s-1}([0,1])$.

Fix $0\leq j<s$. By Leibniz' rule, each term in
$D^j(t^s(1-t)^N)$ is bounded, up to a constant depending only on $s$ and
$j$, by
\[
N^{j-\ell}t^{s-\ell}(1-t)^{N-j+\ell},
\qquad 0\leq\ell\leq j.
\]
For $u,v>0$,
\[
\max_{0\leq t\leq1}t^u(1-t)^v
=\frac{u^uv^v}{(u+v)^{u+v}}
\leq\left(\frac{u}{v}\right)^u.
\]
Taking $u=s-\ell$ and $v=N-j+\ell$ shows that the preceding expression is
$O_{s,j}(N^{j-s})$, uniformly on $[0,1]$, and hence tends to zero since
$j<s$. The same argument applies to $R_N$ after $t\mapsto1-t$. Therefore
\[
\norm{Q_N}_{C^{s-1}([0,1])}\longrightarrow0,
\]
and, after returning to $J$, a sufficiently large $N$ gives the required
polynomial.
\end{proof}

Combining the preceding two lemmas gives the endpoint-sewing step used at each stage of the induction.

\begin{lemma}[Sewing endpoint jets]\label{lem:stitch}
Let $r\geq0$, let $J=[a,b]$ be a nondegenerate compact interval, let $P$ be a polynomial,
and let $\eta>0$. Then one can choose $\delta>0$ such that, whenever real
numbers $\alpha_j,\beta_j$ are prescribed for $0\leq j\leq r+1$ and satisfy
\[
|\alpha_j-P^{(j)}(a)|<\delta,
\qquad
|\beta_j-P^{(j)}(b)|<\delta
\qquad(0\leq j\leq r),
\]
there exists a polynomial $S$ such that
\[
S^{(j)}(a)=\alpha_j,
\qquad
S^{(j)}(b)=\beta_j
\qquad(0\leq j\leq r+1),
\]
and
\[
\norm{S-P}_{C^r(J)}<\eta.
\]
No bound is required on the order-$(r+1)$ data
$\alpha_{r+1},\beta_{r+1}$.
\end{lemma}

\begin{proof}
Apply Lemma~\ref{lem:hermite} to the endpoint errors
\[
u_j=\alpha_j-P^{(j)}(a),
\qquad
v_j=\beta_j-P^{(j)}(b)
\qquad(0\leq j\leq r).
\]
This yields a polynomial $H$ realizing these lower-order jets and satisfying
\[
\norm{H}_{C^r(J)}\leq C\delta.
\]
Choose $\delta>0$ so that $C\delta<\eta/2$, and set
\begin{align*}
A&=\alpha_{r+1}-P^{(r+1)}(a)-H^{(r+1)}(a),\\
B&=\beta_{r+1}-P^{(r+1)}(b)-H^{(r+1)}(b).
\end{align*}
By Lemma~\ref{lem:high-jet}, with $s=r+1$, there is a polynomial $Q$ whose jets of order at most $r$ vanish at both endpoints, whose order-$(r+1)$ endpoint jets are $A$ and $B$, and such that
\[
\norm{Q}_{C^r(J)}<\eta/2.
\]
Then $S=P+H+Q$ has all the prescribed endpoint jets and
\[
\norm{S-P}_{C^r(J)}
\leq \norm{H}_{C^r(J)}+\norm{Q}_{C^r(J)}
<\eta.
\]
\end{proof}

The only arithmetic ingredient needed in the sewing step is the density of the coefficient field. Throughout this section, we adopt the convention $\deg 0=-\infty$.

\begin{lemma}[Approximation over a dense subfield]\label{lem:dense-field}
Let $F\subset\R$ be a dense subfield, let $P\in\R[x]$, let $J$ be a compact
interval, let $r,N_0\geq0$, and let $\delta>0$. Then, for every integer
\[
N>\max\{\deg P,N_0\},
\]
there exists a polynomial $Q\in F[x]$ of degree exactly $N$ such that
\[
\norm{Q-P}_{C^r(J)}<\delta.
\]
In particular, one may choose $Q$ with $\deg Q\geq N_0$.
\end{lemma}

\begin{proof}
Fix an integer $N>\max\{\deg P,N_0\}$. Since the space of polynomials of degree at most $N-1$ is
finite-dimensional and $N>\deg P$, coefficient approximation gives
a polynomial $\widetilde Q\in F[x]$ of degree at most $N-1$ such that
\[
\|\widetilde Q-P\|_{C^r(J)}<\frac{\delta}{2}.
\]
Choose $0\neq c\in F$ so small that
\[
\|cx^N\|_{C^r(J)}<\frac{\delta}{2}.
\]
Then
\[
Q:=\widetilde Q+cx^N\in F[x]
\]
has degree exactly $N$ and
\[
\|Q-P\|_{C^r(J)}
\leq
\|\widetilde Q-P\|_{C^r(J)}
+\|cx^N\|_{C^r(J)}
<\delta.
\]
\end{proof}

\subsection*{The complementary intervals}

Let $E\subset\R$ be nonempty, compact, perfect, and nowhere dense, and set
\[
a:=\min E,
\qquad
b:=\max E.
\]
The connected components of $[a,b]\setminus E$ are pairwise disjoint open
intervals. Since each contains a rational number, they form a countable
family, which we enumerate as
\[
[a,b]\setminus E=\bigcup_{j\geq1}U_j.
\]
No two distinct gaps $U_j$ share an endpoint, since a common endpoint would
be isolated in $E$. For the same reason, no gap has $a$ or $b$ as an
endpoint.

Set $\mathcal J_0:=\{[a,b]\}$. After removing the first $n$ gaps, the
remaining set is the union of $n+1$ pairwise disjoint nondegenerate compact
intervals:
\begin{equation}\label{eq:Jn}
[a,b]\setminus\bigcup_{j=1}^n U_j
=\bigcup_{J\in\mathcal J_n}J,
\qquad
|\mathcal J_n|=n+1.
\end{equation}
Moreover, $U_{n+1}$ lies in the interior of a unique member of
$\mathcal J_n$ and splits it into two nondegenerate compact intervals.

We shall use the fact that gaps of arbitrarily large index occur near every
point of $E$.

\begin{lemma}[Local abundance of gaps]\label{lem:gaps}
Let $x\in E$, let $V$ be an open interval containing $x$, and let $N\geq1$.
Then there exists $j\geq N$ such that
\[
\overline{U_j}\subset V.
\]
In fact, $V$ contains the closures of infinitely many complementary
components of $\R\setminus E$.
\end{lemma}

\begin{proof}
Choose $u,v\in\R\setminus E$ such that
\[
u<x<v,
\qquad
[u,v]\subset V.
\]
This is possible because $E$ is nowhere dense.

Suppose, to the contrary, that only finitely many components of
$\R\setminus E$ have closure contained in $V$.  Every component of
$\R\setminus E$ meeting $[u,v]$ is then either one of these finitely
many components or one of the two components containing $u$ and $v$.
Hence $[u,v]\setminus E$ is a finite union of intervals.

It follows that $[u,v]\cap E$ is a finite union of closed intervals.
Since $E$ has empty interior, each of these intervals is degenerate.
Thus $[u,v]\cap E$ is finite.

But $x\in(u,v)\cap E$, so this would make $x$ an isolated point of $E$,
contrary to the fact that $E$ is perfect.  Therefore every neighborhood
of $x$ contains the closures of infinitely many complementary gaps.

Since only finitely many gaps have index less than $N$, one of these
gaps is $U_j$ for some $j\ge N$, and hence
\[
\overline{U_j}\subset V.
\]
\end{proof}

\subsection*{The arithmetic K\"orner theorem}

We can now assemble the interpolation lemmas into the main technical result of the paper. Besides its application to Liouville numbers below, the theorem gives an independent arithmetic refinement of K\"orner's construction by allowing the local polynomial pieces to lie in any prescribed dense subfield.

\begin{theorem}[Arithmetic polynomial sewing]\label{thm:sewing}
Let $E\subset\R$ be nonempty, compact, perfect, and nowhere dense, and let
$F\subset\R$ be a dense subfield. Then there exists a function
$g\in C_c^\infty(\R)$ such that
\begin{enumerate}[label=\textup{(\roman*)}]
\item for every connected component $I$ of $\R\setminus E$, there exists
$P_I\in F[x]$ such that
\[
g|_I=P_I|_I;
\]
\item $g$ is real-analytic on $\R\setminus E$;
\item $g$ is not real-analytic at any point of $E$.
\end{enumerate}
Consequently,
\[
\NA(g)=E.
\]
\end{theorem}

\begin{proof}
Write $a=\min E$ and $b=\max E$, and enumerate the bounded gaps as
$U_1,U_2,\dots$. We construct functions $g_n$ and polynomials
$Q_1,\dots,Q_n\in F[x]$ inductively so that
\begin{enumerate}[label=\textup{(A\arabic*)},leftmargin=3.1em]
\item $g_n\in C^n(\R)$ and $g_n=0$ on $\R\setminus[a,b]$;
\item $g_n|_{U_j}=Q_j|_{U_j}$ for $1\leq j\leq n$;
\item for every $J\in\mathcal J_n$, the restriction $g_n|_J$ is a real polynomial;
\item $\deg Q_1<\cdots<\deg Q_n$;
\item
\begin{equation}\label{eq:increment}
\norm{g_{n+1}-g_n}_{C^n([a,b])}\leq2^{-n-2}
\qquad(n\geq0).
\end{equation}
\end{enumerate}
For $n=0$, condition \textup{(A4)} is vacuous, and we set $g_0=0$.

Assume that $g_n$ has been constructed. Write
\[
U_{n+1}=(c,d).
\]
This gap lies in the interior of a unique interval
$J_*\in\mathcal J_n$, and on $J_*$ we have $g_n=P_*$ for some real
polynomial $P_*$. Set
\[
\eta_n:=2^{-n-2}.
\]
We shall choose the polynomial $Q_{n+1}\in F[x]$ to be assigned to the new gap
$U_{n+1}$, close to $P_*$ through order $n$ on $[c,d]$, while requiring its
degree to exceed those of all previously assigned gap polynomials.

Let $K=[u,v]\in\mathcal J_{n+1}$, let $J\in\mathcal J_n$ be its unique
parent, and write $g_n|_J=P_J$. The polynomial to be placed on $K$ must
match, through order $n+1$, the piece lying across each endpoint. That
neighboring piece is one of:
\begin{itemize}
\item the zero polynomial, at $a$ or $b$;
\item a previously assigned polynomial $Q_j$, at an old gap endpoint;
\item the new polynomial $Q_{n+1}$, at $c$ or $d$.
\end{itemize}
Denote the corresponding target jets at $u$ and $v$ by
\[
\alpha_0,\dots,\alpha_{n+1},
\qquad
\beta_0,\dots,\beta_{n+1}.
\]
At every old endpoint, the fact that $g_n\in C^n(\R)$ gives
\[
\alpha_j=P_J^{(j)}(u),
\qquad
\beta_j=P_J^{(j)}(v)
\qquad(0\leq j\leq n).
\]
At the new endpoints $c$ and $d$, the discrepancies in these lower-order
jets are bounded by
\[
\norm{Q_{n+1}-P_*}_{C^n([c,d])}.
\]

Apply Lemma~\ref{lem:stitch}, with $r=n$ and tolerance $\eta_n$, to each
of the finitely many intervals in $\mathcal J_{n+1}$. The corresponding
threshold depends only on the parent polynomial, the interval, and
$\eta_n$, and, crucially, not on the prescribed order-$(n+1)$ endpoint
data. Taking the minimum of these finitely many thresholds gives
$\delta_n>0$ such that all the required stitchings are possible whenever
\begin{equation}\label{eq:Qclose}
\norm{Q_{n+1}-P_*}_{C^n([c,d])}
<\min\{\delta_n,\eta_n\}.
\end{equation}
By Lemma~\ref{lem:dense-field}, choose $Q_{n+1}\in F[x]$ satisfying
\eqref{eq:Qclose}, with degree greater than $\deg P_*$ and $n+1$, and greater than the degree of every previously assigned polynomial $Q_j$.

For each $K\in\mathcal J_{n+1}$, Lemma~\ref{lem:stitch} now yields a real
polynomial $S_K$ having exactly the prescribed endpoint jets through order
$n+1$ and satisfying
\[
\norm{S_K-P_J}_{C^n(K)}<\eta_n,
\]
where $J$ is the parent of $K$. Define
\[
g_{n+1}(x)=
\begin{cases}
0, & x\notin[a,b],\\
Q_j(x), & x\in U_j,\quad 1\leq j\leq n+1,\\
S_K(x), & x\in K,\quad K\in\mathcal J_{n+1}.
\end{cases}
\]
The endpoint jets agree through order $n+1$, hence
$g_{n+1}\in C^{n+1}(\R)$. The old gap pieces are unchanged; on the new
gap, \eqref{eq:Qclose} controls $g_{n+1}-g_n$; and on each residual
interval, the same control follows from the stitching estimate. Therefore
\eqref{eq:increment} holds, and the induction is complete.

Fix $r\geq0$. For $n\geq r$, \eqref{eq:increment} together with the fact that $g_{n+1}-g_n$ vanishes outside $[a,b]$, gives
\[
\norm{g_{n+1}^{(r)}-g_n^{(r)}}_{L^\infty(\R)}=\norm{g_{n+1}^{(r)}-g_n^{(r)}}_{L^\infty([a,b])}
\leq2^{-n-2}.
\]
Thus $(g_n^{(r)})_{n\ge r}$ is uniformly Cauchy on $\mathbb R$; denote its limit by $h_r$. On every compact interval, the tails
\[
(g_n^{(r)})_{n\ge r+1}
\qquad\text{and}\qquad
(g_n^{(r+1)})_{n\ge r+1}
\]
converge uniformly to $h_r$ and $h_{r+1}$, respectively. Moreover,
\[
\bigl(g_n^{(r)}\bigr)'=g_n^{(r+1)}
\qquad (n\ge r+1).
\]
The standard theorem on uniform convergence of derivatives therefore gives
\[
h_{r+1}=h_r'.
\]
Induction on $r$ shows that $g:=h_0$ belongs to $C^\infty(\R)$ and that
$g^{(r)}=h_r$ for every $r\geq0$. Since every $g_n$ vanishes outside
$[a,b]$, the same is true of $g$, so $g\in C_c^\infty(\R)$.

Fix a gap $U_j$. The polynomial $Q_j$ assigned to $U_j$ is left unchanged from stage $j$ onward,
and therefore
\[
g|_{U_j}=Q_j|_{U_j},
\qquad
Q_j\in F[x].
\]
On the two unbounded components of $\R\setminus E$, the function $g$ is
identically zero. This proves \textup{(i)} and shows that $g$ is
real-analytic on $\R\setminus E$.

It remains to prove non-analyticity on $E$. Suppose, to the contrary,
that $g$ is real-analytic at some $x\in E$. Then $g$ is real-analytic on
some connected open interval $V\ni x$. By Lemma~\ref{lem:gaps}, there
exist two distinct gaps $U_j,U_k$ whose closures are contained in $V$.
Since $g=Q_j$ on $U_j$, the real-analytic identity theorem gives
$g=Q_j$ throughout $V$. In particular, $Q_j=Q_k$ on $U_k$, and hence
$Q_j=Q_k$ as polynomials. This contradicts
$\deg Q_j\neq\deg Q_k$. Thus $g$ is not real-analytic at any point of
$E$, proving \textup{(iii)} and completing the proof.
\end{proof}

The coefficient restriction in Theorem~\ref{thm:sewing} already yields a useful arithmetic preservation property.

\begin{corollary}[Arithmetic preservation]\label{cor:sewing-arithmetic}
Under the hypotheses of Theorem~\ref{thm:sewing}, suppose in addition that
$E\cap F=\varnothing$. Then
\[
D^m g(F)\subseteq F
\qquad (m\geq0).
\]
In particular, if $F=\Q$ and $E\cap\Ralg=\varnothing$, then for every real
number field $K\subset\R$,
\[
D^m g(K)\subseteq K
\qquad (m\geq0).
\]
\end{corollary}
\begin{proof}
If $\alpha\in F$, then $\alpha\notin E$, so $\alpha$ belongs to a component
$I$ of $\R\setminus E$ on which $g=P_I$ for some $P_I\in F[x]$. Hence
\[
D^m g(\alpha)=P_I^{(m)}(\alpha)\in F.
\]
For the final assertion, if $\alpha\in K$, then $\alpha\in\Ralg$ and hence
$\alpha\notin E$; since $P_I\in\Q[x]$, the same argument gives
$D^m g(\alpha)\in K$.
\end{proof}

\begin{remark}\label{rem:korner}
Theorem~\ref{thm:sewing} is not a formal restatement of Körner's theorem.  The dense-field condition is arithmetic.  Its viability rests on an asymmetric induction: at each stage, only the polynomial assigned to the newest gap is fixed permanently.  Its $F$-polynomial can approximate the previous real piece in all derivatives already controlled, while the remaining residual intervals absorb the new highest-order mismatch by Lemma~\ref{lem:high-jet}.
\end{remark}

\section{Anchored rational germs and arithmetic dynamics}\label{sec:germs}

The polynomial sewing theorem provides exact local arithmetic models.
To use them along iterates, however, their rational extensions must remain
nonconstant under composition and differentiation. A finite complex pole
provides a convenient and robust certificate of this property.

\begin{definition}
A rational function $R\in\Q(z)$ will be called \emph{anchored} if
\[
 R(\infty)=\infty
\]
on the Riemann sphere and $R$ has at least one finite pole.
\end{definition}

The usefulness of this notion is that anchoredness is preserved under
composition, while differentiation preserves a finite pole and hence
nonconstancy.

\begin{lemma}[Persistence of finite poles]\label{lem:anchors}
Let $R_1,\dots,R_n\in\Q(z)$ be anchored.  Then
\[
 S=R_n\circ\cdots\circ R_1
\]
is anchored.  Moreover, for every $m\geq0$, the rational derivative $S^{(m)}$ has a finite pole and is nonconstant.
\end{lemma}

\begin{proof}
Clearly $S\in\Q(z)$ and $S(\infty)=\infty$.  We prove the finite-pole assertion by induction on $n$.  For $n=1$ it is part of the definition.  Write
\[
T=R_{n-1}\circ\cdots\circ R_1.
\]
Since each $R_j$ is anchored, it is nonconstant; hence $T$ is a
nonconstant rational function. Therefore
\[
T:\widehat{\C}\to\widehat{\C}
\]
is surjective. Let $w\in\C$ be a finite pole of $R_n$. Choose $z\in\widehat\C$ with $T(z)=w$.  Since $T(\infty)=\infty\neq w$, the point $z$ is finite.  Hence $S(z)=\infty$, so $z$ is a finite pole of $S$.

If $S$ has a pole of order $r\geq1$ at $z_0$, its Laurent expansion begins with
\[
 c(z-z_0)^{-r},\qquad c\neq0.
\]
The $m$th derivative begins with
\[
 c(-1)^m r(r+1)\cdots(r+m-1)(z-z_0)^{-r-m}
\]
(with the usual interpretation for $m=0$).  Thus every $S^{(m)}$ has a pole at $z_0$ and cannot be constant.
\end{proof}

The following rational perturbation provides a common pair of finite poles for all local models.

\begin{lemma}[A universal anchored local model]\label{lem:model}
Let $P\in\Q[z]$ and let $\varepsilon\in\Q\setminus\{0\}$. Define
\[
R_{P,\varepsilon}(z)
:=z+\varepsilon\left(P(z)+\frac1{1+z^2}\right).
\]
Then $R_{P,\varepsilon}\in\Q(z)$ and its finite poles are precisely $i$ and $-i$.
Moreover, if $z+\varepsilon P(z)$ is nonconstant, then $R_{P,\varepsilon}$ is anchored.
\end{lemma}

\begin{proof}
We have
\[
R_{P,\varepsilon}(z)
=
\frac{(z+\varepsilon P(z))(1+z^2)+\varepsilon}{1+z^2}.
\]
At $z=\pm i$, the numerator equals $\varepsilon\neq0$, so neither pole is
cancelled. Since $1+z^2$ has no other zeros, these are exactly the finite
poles of $R_{P,\varepsilon}$.

If $z+\varepsilon P(z)$ is nonconstant, then it is a nonconstant polynomial
and hence tends to $\infty$ on the Riemann sphere as $z\to\infty$, while
$(1+z^2)^{-1}\to0$. Therefore
\[
R_{P,\varepsilon}(\infty)=\infty,
\]
so $R_{P,\varepsilon}$ is anchored.
\end{proof}

We isolate the local-to-global arithmetic principle that will be used for iterates.

\begin{proposition}[Arithmetic germ principle]\label{prop:germ-principle}
Let $f\in C^\infty(\R)$ and let $A\subset\R$.  Assume:
\begin{enumerate}[label=\textup{(\alph*)}]
\item $f(A)\subseteq A$;
\item for every $\alpha\in A$, there is an open interval $I_\alpha\ni\alpha$ and an anchored rational function $R_\alpha\in\Q(z)$ such that
\[
 f|_{I_\alpha}=R_\alpha|_{I_\alpha}.
\]
\end{enumerate}
Then for every $\alpha\in A$ and every $n\geq1$, there are an open interval $V\ni\alpha$ and anchored rational functions $R_0,\dots,R_{n-1}\in\Q(z)$ such that
\[
 f^{\circ n}|_V=(R_{n-1}\circ\cdots\circ R_0)|_V.
\]
Consequently:
\begin{enumerate}[label=\textup{(\roman*)}]
\item if $A=K$ is a real number field, then
\[
 D^m(f^{\circ n})(K)\subseteq K
\qquad(m\geq0,\ n\geq1);
\]
\item if $A=\Lio$, then
\[
 D^m(f^{\circ n})(\Lio)\subseteq\Lio
\qquad(m\geq0,\ n\geq1).
\]
\end{enumerate}
\end{proposition}

\begin{proof}
Fix $\alpha\in A$ and set
\[
 \alpha_j=f^{\circ j}(\alpha),
 \qquad 0\leq j\leq n-1.
\]
By (a), each $\alpha_j$ belongs to $A$.  Choose $I_j=I_{\alpha_j}$ and $R_j=R_{\alpha_j}$ as in (b).  The set
\[
 V=\bigcap_{j=0}^{n-1}(f^{\circ j})^{-1}(I_j)
\]
is open and contains $\alpha$. Replacing $V$ by the connected component
containing $\alpha$, we may assume that $V$ is an open interval. For $x\in V$, every intermediate point $f^{\circ j}(x)$ lies in $I_j$, so successive substitution gives
\[
 f^{\circ n}(x)=R_{n-1}\circ\cdots\circ R_0(x).
\]

Let $S=R_{n-1}\circ\cdots\circ R_0$.  By Lemma~\ref{lem:anchors}, every $S^{(m)}$ is a nonconstant rational function in $\Q(z)$. Since $S$ agrees with the smooth iterate $f^{\circ n}$ on a neighborhood of $\alpha$, neither $S$ nor any $S^{(m)}$ has a pole at $\alpha$. If $\alpha\in K$, then
\[
 D^m(f^{\circ n})(\alpha)=S^{(m)}(\alpha)\in K,
\]
because $S^{(m)}$ has rational coefficients. This proves (i). If $\alpha\in\Lio$, the preceding no-pole observation allows us to apply Theorem~\ref{thm:maillet}, which gives
\[
 D^m(f^{\circ n})(\alpha)=S^{(m)}(\alpha)\in\Lio,
\]
which proves (ii).
\end{proof}

\section{Construction of the diffeomorphism}\label{sec:construction}

We now combine the arithmetic sewing construction with the rational
pole mechanism of Section~\ref{sec:germs} to prove the structural form
of the main result.

\begin{theorem}[Prescribed-singularity theorem]\label{thm:cloaking}
Let $E\subset\R$ be a nonempty compact perfect nowhere-dense set such that
\[
E\cap(\Ralg\cup\Lio)=\varnothing.
\]
For every neighborhood $\mathcal U$ of the identity in the compact-open $C^\infty$ topology, there exists an orientation-preserving $C^\infty$ diffeomorphism $f:\R\to\R$ with the following properties:
\begin{enumerate}[label=\textup{(\roman*)}]
\item $f\in\mathcal U$ and $\NA(f)=E$;
\item for every connected component $I$ of $\R\setminus E$, there exists $R_I\in\Q(z)$ such that $f|_I=R_I|_I$, the finite poles of $R_I$ are exactly $i$ and $-i$, and $R_I(\infty)=\infty$;
\item $f$ is transcendental over $\R(x)$;
\item for every real number field $K\subset\R$, every $n\geq1$, and every $m\geq0$,
\[
D^m(f^{\circ n})(K)\subseteq K;
\]
\item for every $n\geq1$ and every $m\geq0$,
\[
D^m(f^{\circ n})(\Lio)\subseteq\Lio.
\]
\end{enumerate}
Moreover, one may choose a sequence of such diffeomorphisms converging to the identity in $C^\infty(\R)$.
\end{theorem}

Fix $E$ as in Theorem~\ref{thm:cloaking}. Apply Theorem~\ref{thm:sewing} with $F=\Q$. We obtain $g\in C_c^\infty(\R)$ such that $\NA(g)=E$ and, for every component $I$ of $\R\setminus E$,
\begin{equation}\label{eq:g-piece}
g|_I=P_I|_I
\qquad\text{for some }P_I\in\Q[x].
\end{equation}
Let
\[
h(x)=\frac1{1+x^2}.
\]
The role of this rational perturbation is structural: its two nonreal poles $\pm i$ are inherited by every local rational model and provide a persistent certificate of nonconstancy under arbitrary compositions and differentiations. For $\varepsilon\in\Q\setminus\{0\}$, define
\begin{equation}\label{eq:f-eps}
f_\varepsilon(x)=x+\varepsilon(g(x)+h(x)).
\end{equation}

\begin{lemma}[Diffeomorphism and approximation]\label{lem:diffeomorphism}
For every sufficiently small nonzero rational $\varepsilon$, the map $f_\varepsilon$ is an orientation-preserving $C^\infty$ diffeomorphism of $\R$ onto itself. Moreover,
\[
f_\varepsilon\longrightarrow\id
\qquad\text{in }C^\infty(\R)
\]
as $\varepsilon\to0$.
\end{lemma}

\begin{proof}
Since $g$ is compactly supported and $h'$ is bounded,
\[
M:=\norm{g'+h'}_{L^\infty(\R)}<\infty.
\]
If $|\varepsilon|M<1/2$, then
\[
f_\varepsilon'(x)
=1+\varepsilon(g'(x)+h'(x))
\geq1-|\varepsilon|M>\frac12
\]
for every $x\in\R$. Hence $f_\varepsilon$ is strictly increasing and is a local $C^\infty$ diffeomorphism at every point. Since $g(x)=0$ for $|x|$ large and $h(x)\to0$ as $|x|\to\infty$,
\[
f_\varepsilon(x)-x\longrightarrow0
\qquad(x\to\pm\infty).
\]
Thus $f_\varepsilon(x)\to\pm\infty$ as $x\to\pm\infty$. Strict monotonicity now implies that $f_\varepsilon$ maps $\R$ bijectively onto $\R$, and the inverse function theorem shows that its inverse is $C^\infty$.

Finally,
\[
f_\varepsilon-\id=\varepsilon(g+h),
\]
so, for every compact-open seminorm $p_{A,r}$,
\[
p_{A,r}(f_\varepsilon-\id)
=|\varepsilon|p_{A,r}(g+h)\longrightarrow0.
\]
\end{proof}

We shall also impose
\begin{equation}\label{eq:epsilon-small}
|\varepsilon|\norm{g'}_{L^\infty(\R)}<\frac12.
\end{equation}
This ensures that the polynomial $z+\varepsilon P_I(z)$ occurring in
each local model remains nonconstant.

\begin{lemma}[Local rationality and the exact analytic locus]\label{lem:local-rationality}
Assume that $\varepsilon\in\Q\setminus\{0\}$ is small enough for Lemma~\ref{lem:diffeomorphism} and \eqref{eq:epsilon-small}, and put $f=f_\varepsilon$. For every component $I$ of $\R\setminus E$,
\[
f|_I=R_I|_I,
\qquad
R_I(z)=z+\varepsilon\left(P_I(z)+\frac1{1+z^2}\right)\in\Q(z).
\]
Each $R_I$ is anchored and has precisely the finite poles $i$ and $-i$. Furthermore,
\[
\NA(f)=E.
\]
\end{lemma}

\begin{proof}
The local identity follows from \eqref{eq:g-piece} and \eqref{eq:f-eps}. By Lemma~\ref{lem:model}, the only finite poles are $i$ and $-i$. It remains to verify that
\[
A_I(z):=z+\varepsilon P_I(z)
\]
is nonconstant. Otherwise, $1+\varepsilon P_I'\equiv0$, and for every $x\in I$ we would have
\[
1=|\varepsilon|\,|P_I'(x)|
=|\varepsilon|\,|g'(x)|
\leq |\varepsilon|\norm{g'}_{L^\infty(\R)}<\frac12,
\]
a contradiction. Hence $R_I(\infty)=\infty$, so $R_I$ is anchored.

On $\R\setminus E$, the function $f$ is locally rational without real poles and is therefore real-analytic. If $x\in E$ and $f$ were real-analytic in a neighborhood of $x$, then
\[
g=\varepsilon^{-1}(f-\id)-h
\]
would be real-analytic there, contrary to $\NA(g)=E$. Thus $\NA(f)=E$.
\end{proof}

The prescribed non-analyticity locus now forces global transcendence.

\begin{proposition}[Global transcendence]\label{prop:transcendence}
The function $f$ is transcendental over $\R(x)$.
\end{proposition}

\begin{proof}
If $f$ were algebraic over $\R(x)$, then, after clearing denominators, there would exist a nonzero polynomial $P\in\R[X,Y]$ such that $P(x,f(x))=0$ for every $x\in\R$. Theorem~\ref{thm:smooth-algebraic} would make $f$ real-analytic on $\R$, contradicting $\NA(f)=E\neq\varnothing$.
\end{proof}

We next record the arithmetic consequence of the local rational models before passing to iterates.

\begin{proposition}[One-step arithmetic invariance]\label{prop:one-step}
For every real number field $K\subset\R$ and every $m\geq0$,
\[
f^{(m)}(K)\subseteq K,
\qquad
f^{(m)}(\Lio)\subseteq\Lio.
\]
In particular, $f(K)\subseteq K$ and $f(\Lio)\subseteq\Lio$.
\end{proposition}

\begin{proof}
Let $\alpha\in K$. Since $E\cap\Ralg=\varnothing$, we have $\alpha\notin E$, so
$\alpha$ belongs to a component $I$ of $\R\setminus E$. On a neighborhood of
$\alpha$, the map $f$ agrees with $R_I\in\Q(z)$. Therefore
\[
f^{(m)}(\alpha)=R_I^{(m)}(\alpha)\in K,
\]
since $R_I^{(m)}\in\Q(z)$ and differentiation introduces no poles away
from those of $R_I$.

Now let $\xi\in\Lio$. Since $\xi\notin E$, the map $f$ agrees near $\xi$ with
an anchored rational function $R_I\in\Q(z)$. By Lemma~\ref{lem:anchors},
$R_I^{(m)}$ is nonconstant and has a finite pole, while it has no pole at
the real point $\xi$. Hence Maillet's theorem applies and gives
\[
f^{(m)}(\xi)=R_I^{(m)}(\xi)\in\Lio.
\]
\end{proof}

\begin{proof}[Proof of Theorem~\ref{thm:cloaking}]
Fix a neighborhood $\mathcal U$ of $\id$. Lemma~\ref{lem:diffeomorphism} gives $f_\varepsilon\to\id$ in $C^\infty(\R)$ as $\varepsilon\to0$. Choose a nonzero rational $\varepsilon$ so small that
\[
f_\varepsilon\in\mathcal U,
\qquad
|\varepsilon|\norm{g'+h'}_{L^\infty(\R)}<\frac12,
\qquad
|\varepsilon|\norm{g'}_{L^\infty(\R)}<\frac12.
\]
Put $f=f_\varepsilon$.

Lemma~\ref{lem:diffeomorphism} gives an orientation-preserving diffeomorphism. Lemma~\ref{lem:local-rationality} gives the exact non-analyticity locus and the anchored local models with poles $i$ and $-i$. Proposition~\ref{prop:transcendence} gives global transcendence.

By Proposition~\ref{prop:one-step}, every real number field $K$ and the set
$\Lio$ are invariant under $f$, while Lemma~\ref{lem:local-rationality}
provides the required anchored rational germs at each point of these sets.
Therefore Proposition~\ref{prop:germ-principle} applies to both $K$ and
$\Lio$, and yields
\[
D^m(f^{\circ n})(K)\subseteq K,
\qquad
D^m(f^{\circ n})(\Lio)\subseteq\Lio
\]
for every $n\geq1$ and every $m\geq0$.

Finally, choosing nonzero rationals $\varepsilon_\nu\to0$ satisfying the same two smallness inequalities produces a sequence of maps with all the asserted properties converging to the identity in $C^\infty(\R)$.
\end{proof}

\begin{proof}[Proof of Theorem~\ref{thm:headline}]
Choose, for instance, $\eta=1/2$ in Proposition~\ref{prop:invisible}, and let
$E\subset(0,1)$ be the resulting nonempty compact perfect nowhere-dense set.
Then
\[
E\cap(\Ralg\cup\Lio)=\varnothing.
\]
Applying Theorem~\ref{thm:cloaking} to this $E$ yields an
orientation-preserving $C^\infty$ diffeomorphism, arbitrarily close to the
identity and transcendental over $\R(x)$, for which
\[
D^m(f^{\circ n})(K)\subseteq K,
\qquad
D^m(f^{\circ n})(\Lio)\subseteq\Lio
\]
for every real number field $K\subset\R$, every $n\geq1$, and every $m\geq0$.
This is precisely the conclusion of Theorem~\ref{thm:headline}.
\end{proof}

\section{Concluding remarks}\label{sec:concluding}

The construction rests on a separation of roles: smoothness and prescribed
non-analyticity are handled by the sewing mechanism, while arithmetic
preservation is ultimately local, through exact rational models and
Maillet's theorem. In this way, global height control is replaced by exact
local rationality.

\medskip

\noindent
\emph{Geometry of the non-analyticity locus.}
The prescribed-singularity form of the main theorem leaves considerable
freedom in the geometry of the exceptional set. In particular,
Proposition~\ref{prop:invisible} and Theorem~\ref{thm:cloaking} show that,
for every $0<\eta<1$, the diffeomorphism may be chosen so that
\[
\lambda\bigl(\NA(f)\cap(0,1)\bigr)>1-\eta.
\]
Thus non-analyticity may occupy almost all of a fixed interval in measure,
while remaining completely disjoint from both the real algebraic and
Liouville numbers. More generally, the construction applies to any nonempty
compact perfect nowhere-dense set $E$ avoiding these two arithmetic sets.

\medskip

\noindent
\emph{The analytic barrier.}
The mechanism is genuinely specific to the smooth category. Its flexibility
comes from allowing distinct rational germs on different components of
$\R\setminus E$, with the failure of analyticity confined to $E$. In the
real-analytic, and a fortiori in the entire, category, the identity principle
prevents such independent local models from being sewn together.
Consequently, the present method does not yield a route to Mahler's original
question for transcendental entire functions; overcoming this analytic
rigidity appears to require a substantially different mechanism.

\medskip

\noindent
\emph{Beyond Liouville numbers.}
The rational-germ principle of Proposition~\ref{prop:germ-principle} is not
specific to $\Lio$. More generally, let $A\subset\R$ be disjoint from the
prescribed non-analyticity locus and stable under nonconstant rational
functions in $\Q(z)$ wherever they are defined. The same local-rationality
mechanism then yields preservation of $A$ under the corresponding iterates
and derivatives. This suggests that arithmetic sewing may be useful for
other rationally invariant Diophantine classes, whenever a suitable
geometric separation from the non-analyticity locus is available.

\section*{Funding}

This work was supported by the National Council for Scientific and Technological Development (CNPq), under
Grant No. 304467/2023-5.

\end{document}